\documentclass[11pt]{article}
\usepackage{amsmath,amssymb,latexsym}
\usepackage{graphicx}
\def\Z{{\mathbb Z}}
\def\bZ{{\mathbb Z}}

\def\GL{{\rm GL}}
\def\SL{{\rm SL}}
\def\Cl{{\rm Cl}}
\def\Sym{{\rm Sym}}
\def\Det{{\rm Det}}
\def\Aut{{\rm Aut}}
\def\Stab{{\rm Stab}}

\def\End{{\rm End}}

\def\Frac{{\rm Frac}}

\def\im{{\rm im}}
\def\P{{\mathbb P}}
\def\Disc{{\rm Disc}}

\def\red{{\rm red}}

\def\R{{\mathbb R}}
\def\C{{\mathbb C}}
\def\F{{\mathbb F}}
\def\bR{{\mathbb R}}

\def\bF{{\mathbb F}}

\def\I{{\mathcal I}}
\def\Q{{\mathbb Q}}

\def\Z{{\mathbb Z}}
\def\P{{\mathbb P}}
\def\F{{\mathbb F}}
\def\Q{{\mathbb Q}}
\def\O{{\mathcal O}}
\def\cI{{\mathcal I}}
\def\bQ{{\mathbb Q}}

\def\bZ{{\mathbb Z}}
\def\bP{{\mathbb P}}
\def\bF{{\mathbb F}}
\def\bQ{{\mathbb Q}}
\def\cO{{\mathcal O}}
\def\cI{{\mathcal I}}

\newcommand{\ritem}[1]{\item[{\rm #1}]}
\newtheorem{theorem}{Theorem}
\newtheorem{lemma}[theorem]{Lemma}
\newtheorem{corollary}[theorem]{Corollary}

\newtheorem{remark}[theorem]{Remark}
\newenvironment{proof}{\noindent {\bf Proof:}}{$\Box$ \vspace{2 ex}}

\begin{document}

\title{On 
the mean number of 
2-torsion elements in the class groups,
narrow class groups,
and ideal groups 
 of  cubic orders and fields}

\author{Manjul Bhargava and Ila Varma}
%\address{Princeton University}

\maketitle

\begin{abstract}
Given any family of cubic fields defined by local conditions at finitely many primes, 
  we determine the mean number of 2-torsion elements in the class
  groups and narrow class groups of 
  these 
  cubic fields, when they are ordered by their absolute discriminants.

For
   an order $\O$ in a cubic field, we 
   study the three
  groups: $\Cl_2(\O)$, the group of ideal classes
  of $\O$ of order 2; 
   $\Cl^+_2(\O)$, the group of narrow ideal classes
  of $\O$ of order 2; 
  and~$\I_2(\O)$, the group of ideals of $\O$ of order 2.
  We prove
  that the mean value of the difference $|\Cl_2(\O)|-\frac{1}{4}|\I_2(\O)|$ is always equal to 1, regardless
  of whether one averages over the maximal orders in real cubic
  fields, over all orders in real cubic fields, or indeed over {\it any}
  family of real cubic orders defined by local conditions.  For the narrow
  class group, we prove that the average value of the difference 
  $|\Cl^+_2(\O)|-|\I_2(\O)|$ is equal to 1 for any such family.  Also, for any family
  of complex cubic orders defined by local conditions, we prove
  similarly that the mean value of the difference $|\Cl_2(\O)|-\frac{1}{2}|\I_2(\O)|$ is always equal to 1,
  independent of the family.

  The determination of these mean numbers allows us to prove a number
  of further results as by-products.  Most notably, we prove---in
  stark contrast to the case of quadratic fields---that: 1) a positive proportion of cubic fields
 have {\it odd} class number; 
  2) a positive proportion of real cubic fields have
  isomorphic 2-torsion in the class group and
  the narrow class group; and 3) a positive proportion of real
  cubic fields contain units of mixed real signature.
  We also show that a positive proportion of real cubic fields have narrow class group strictly larger than the class
 group, and thus a positive proportion of real cubic fields do not possess units of every possible real
  signature.
\end{abstract}

\section{Introduction}

For an order $\O$ in a number field $K$, let $\Cl(\O)$ and $\Cl^+(\O)$
denote the class group and the narrow class group\footnote{Recall that
  the {\it class group} $\Cl(\O)$ of an order $\O$ in a number field
$K$ is defined as $\cI(\O)
  / P(\O),$ where $\cI(\O)$ is the group of invertible fractional ideals
  of $\O$ and $P(\O)$ is the group of principal fractional ideals of
  $\O$, that is, ideals of the form $a\O$ where $a\in K^\times$.  The
  {\it narrow class group} $\Cl^+(\O)$ is defined as the quotient $\cI(\O)
  / P(\O)^+$, where now $P(\O)^+$ is the group of \emph{totally positive}
  principal fractional ideals of $\O$, that is, ideals of the form $a\O$
  where $a$ is an element of $K^\times$ such that $\sigma(a)$ is
  positive for every embedding $\sigma : K \to \R$.  
} of $\O$
respectively.  If $\O$ is the maximal order of $K$, then one defines
the class group and narrow class group of $K$ by $\Cl(K):=\Cl(\O)$ and
$\Cl^+(K):=\Cl^+(\O)$.  Finally, for any prime $p$ let $\Cl_p(\O)$, $\Cl_p^+(\O)$,
$\Cl_p(K)$, and $\Cl_p^+(K)$ denote the $p$-torsion subgroups of
$\Cl(\O)$, $\Cl^+(\O)$, $\Cl(K)$, and $\Cl^+(K)$, respectively.

In this article, we begin by proving: 

\begin{theorem}\label{cgvsncg}
When the set of isomorphism classes of cubic fields satisfying any specified set of local conditions at any finite set of primes are ordered by their absolute discriminants:
\begin{itemize} 
\item[\rm{(a)}]
The average number of $2$-torsion elements in the class groups of such totally real cubic fields is $5/4$.
\item[\rm{(b)}]
The average number of $2$-torsion elements in the $($\!narrow$)$ class groups of such  complex cubic fields is $3/2$.
\item[\rm{(c)}]
The average number of $2$-torsion elements in the narrow class groups of such totally real cubic fields is $2$.
\end{itemize}
\end{theorem}
Thus, the average number of 2-torsion elements in the class groups or narrow class groups of cubic fields of given real signature remains the same, regardless of which family of cubic fields we are averaging over.

In the case where no local specifications are made in Theorem~\ref{cgvsncg} (i.e., where we average over all cubic fields), parts (a) and (b) of Theorem~\ref{cgvsncg} were first proven in \cite{density} by
rather indirect methods, and confirmed a special
case of the Cohen--Martinet--Malle heuristics (see \cite{CM,Malle}). In the current paper, we
generalize these results using a more direct correspondence (see Theorem~\ref{potqfideal}), which additionally allows us to obtain Theorem~\ref{cgvsncg}(c), and indeed Theorems~\ref{cgvsncg}(a)--(c) for general families of cubic fields defined by local conditions at finitely many primes.\footnote{We are not aware of any extensions of the
  Cohen--Lenstra--Martinet--Malle heuristics that may have led to the predictions
  of Theorem~\ref{cgvsncg}(c).  
  It would be interesting to have similar
  Cohen--Lenstra--Martinet--Malle style heuristics that apply more generally to
  narrow class groups as well as to orders.  
  }
Our current method is more direct in the sense that, not only does it
yield a count of all the 2-torsion elements in class groups and narrow class groups 
in various
families of cubic fields, but it in fact gives an explicit
construction of these elements. 

This direct correspondence also allows us to study on average how many ideal classes of order~2 in the class groups of cubic orders arise simply because there exists an 
ideal of order~2 representing~it. Recall that the {\it ideal group} $\cI(\cO)$ of an order $\cO$ is the group of invertible fractional ideals of $\cO$, of which the class group $\Cl(\cO)$ is a quotient. 
If an order $\cO$ is not maximal, then 
it is possible for an \emph{ideal} $I$ of $\cO$ to have order 2, i.e., $I$ can be a fractional ideal of $\cO$ satisfying $I^2 = \cO$ but $I \neq \cO$.\footnote{For maximal orders, it is easy to show that any 2-torsion element (indeed, any torsion element) in the ideal group must be the trivial ideal.} 

For an order $\cO$ in a number field, let $\cI_p(\cO)$ denote the $p$-torsion subgroup of the ideal group~$\cI(\cO)$ of~$\cO$. In \cite[Thm.~8]{clorders}, we showed that the mean value of the difference $|\Cl_3(\cO)| - |\cI_3(\cO)|$ is always equal to 1, regardless of whether one averages over the maximal orders in complex quadratic fields, over all orders in such fields, or even over very general families of complex quadratic orders defined by local conditions. Similarly, for very general families of real quadratic orders defined by local conditions, we showed that the mean value of the difference $|\Cl_3(\cO)| - \frac{1}{3}|\cI_3(\cO)|$ is always equal to 1, independent of the family. 

In this article, we prove the analogues of these results for the
$2$-torsion elements in the class groups and ideal groups of {\it cubic} orders.
To state the result, we require just a bit of notation. For each prime $p$, let $\Sigma_p$ be any set of isomorphism classes of orders in \'etale cubic algebras over $\bQ_p$. We say that the collection $(\Sigma_p)$ of local specifications is \emph{acceptable} if, for all sufficiently large $p$, the set $\Sigma_p$ contains all maximal cubic rings over $\bZ_p$,
or at least all those maximal cubic rings over $\bZ_p$ that are not totally ramified at $p$. Then we prove: %Let $

\begin{theorem}\label{diff}
Let $(\Sigma_p)$ be any acceptable collection of local specifications as above, and let $\Sigma$ denote the set of all isomorphism classes of cubic orders $\cO$ such that $\cO \otimes \bZ_p \in \Sigma_p$ for all $p$. Then, when orders in $\Sigma$ are ordered by their absolute discriminants:
\begin{itemize} 
\item[\rm{(a)}]
The average size of $|\Cl_2(\O)|-\frac{1}{4}|\I_2(\O)|$ for totally real cubic orders
$\O$ in $\Sigma$ is $1$.
\item[\rm{(b)}]
The average size of $|\Cl_2(\O)|-\frac{1}{2}|\I_2(\O)|=|\Cl_2^+(\O)|-\frac{1}{2}|\I_2(\O)|$ for complex cubic orders
$\O$ in~$\Sigma$ is~$1$.
\item[\rm{(c)}]
The average size of $|\Cl^+_2(\O)|-|\I_2(\O)|$ for totally real cubic orders
$\O$ in $\Sigma$ is $1$.
\end{itemize}
\end{theorem}

It is rather intriguing that the mean values in Theorem~\ref{diff}  
remain the same regardless of the choice $(\Sigma_p)$ of acceptable local specifications for the cubic orders being averaged over. If $(\Sigma_p)$ is acceptable and each $\Sigma_p$ consists only of {\it maximal} cubic rings over $\Z_p$, then we recover (a more general version of) Theorem~\ref{cgvsncg}, since the only 2-torsion element of the ideal group of a maximal order is the trivial~ideal.

Theorem~\ref{cgvsncg} immediately implies that most cubic fields
have odd class number.  More precisely:

\begin{corollary}\label{odd} 
\begin{itemize}
\item[$(1)$]
A positive proportion $($at least $75\%)$ of totally real cubic fields have odd class number.
\item[$(2)$]
A positive proportion $($at least $50\%)$ of complex cubic fields have odd class number.
\end{itemize}
\end{corollary}

We now compare the class group with the narrow class
group.  For orders $\O$ in number fields with predominantly real
places, such as totally real number fields, it is natural to expect
the narrow class group to be strictly larger than the class group,
since one is forming a quotient of $I_\O$ by principal ideals
satisfying much more stringent conditions.  In the quadratic case,
this expectation is supported by:

\begin{theorem} 
  When ordered by their discriminants, a density of $100\%$ of real
  quadratic fields~$F$ satisfy $\Cl(F) \ne \Cl^+(F)$.\footnote{This is
    not to say that {\it all} real quadratic fields satisfy $\Cl(F)
    \ne \Cl^+(F)$; indeed, there are are also infinitely many
    quadratic fields for which $\Cl(F)=\Cl^+(F)$, but asymptotically
    such fields have density zero in the set of all quadratic fields.
The 0\% of quadratic fields satisfying
  $\Cl(F)=\Cl^+(F)$ are actually very interesting; see, for example,
  the compelling heuristics of Stevenhagen~\cite{S} on such fields and the
  beautiful recent work of Fouvry and Kl\"uners~\cite{FK} in this regard.}
\end{theorem}

In fact, since in the real quadratic case the size of the narrow class
group is either the same as or twice the size of the class group, 
it is natural to try and distinguish these two groups 
just based on the sizes of their 2-torsion subgroups.  
We have:

\begin{theorem} \label{quad2}
  When ordered by their discriminants, a density of $100\%$ of real
  quadratic fields~$F$ satisfy $\Cl_2(F) \ne \Cl_2^+(F)$.
\end{theorem}

Now recall that for any number field     
$L$ with $r$ distinct real embeddings, there is a         
{\it signature} homomorphism $\mbox{sgn}: \cO_L^\times \rightarrow \{\pm 1\}^r$ that takes    
a {\it unit} of $L$---i.e., an invertible element of the ring of integers $\cO_L$ of $L$---to its sign (one sign for each of the real embeddings $\sigma: L \rightarrow     
\bR$); a unit $u$ has \emph{mixed signature} if $\mbox{sgn}(u) 
\neq \mbox{sgn}(\pm1)$,
i.e., if there exist two      
real embeddings $\sigma_{+},\sigma_{-}: L \rightarrow \bR$ such that                    
$\sigma_{\pm}(u) = \pm 1$.   Then, in terms of units of quadratic fields, we have:                                                         
                                                                                        
\begin{corollary}\label{quadmixed}
When ordered by their discriminants, a density of $0\%$ of real
quadratic fields have a unit of mixed signature.
\end{corollary}

Thus, at least asymptotically, we completely understand the question
of how often the class group and narrow class group differ for real
quadratic fields. Namely, 100\% of the time, the narrow class group
looks like the class group with one extra 2-torsion
factor.

Theorems 4 and 5 are well-known and easily proven; they are true
primarily because for most~$d$ there are local obstructions to having
a solution to the negative norm unit equation $x^2-dy^2=-4$.  Also,
the 2-torsion elements in class groups and narrow class groups in the
quadratic case are mostly governed by genus theory.  One of the
reasons Cohen and Lenstra did not formulate heuristics for narrow
class groups in the quadratic case is that these groups differ from
class groups of quadratic fields only in their 2-parts, which are
related to genus theory.
From this point of view, when computing narrow class groups vs.\ class
groups, it is natural to try and investigate a case where 2 is
unrelated to genus theory---for example, the case of {\it cubic}
fields, to which we now turn.

A classical theorem of Armitage and Frohlich~\cite{AF} states that
\begin{equation*}\label{afeq}
\dim_{\Z/2\Z}(\Cl_2^+(K))-\dim_{\Z/2\Z}(\Cl_2(K))\leq 1
\end{equation*}
 for cubic fields $K$.  Thus
Theorems~\ref{cgvsncg}(a) and~\ref{cgvsncg}(c), when taken together, immediately
imply:

\begin{corollary}\label{nocontrast}
  A positive proportion $($at least $50\%)$ of totally real cubic
  fields $K$ satisfy $\Cl_2(K)\neq \Cl_2^+(K)$.
\end{corollary}

\begin{corollary}\label{contrast}
  A positive proportion $($at least $25\%)$ of totally real cubic
  fields $K$ satisfy $\Cl_2(K)= \Cl_2^+(K)$.
\end{corollary}
Note that Corollary~\ref{contrast} for cubic fields is in stark
contrast to Theorem~\ref{quad2} for quadratic fields.  

As far as the existence of units of mixed type in cubic fields, we
have:

\begin{corollary}\label{mixed1}
A positive proportion $($at least $50\%)$ of totally real cubic fields
$K$ do not possess units of every possible signature. 
\end{corollary}

\begin{corollary}\label{mixed2}
A positive proportion $($at least $75\%)$ of totally real cubic fields
$K$ possess units of mixed signature. 
\end{corollary}
Note again the stark contrast between Corollary~\ref{mixed2} and the
corresponding situation in Corollary~\ref{quadmixed} for quadratic
fields.  
Corollaries~\ref{odd} and \ref{nocontrast}--\ref{mixed2} continue to hold, with the same percentages, even
when one averages over just those cubic fields satisfying any desired set of local conditions at a finite set of primes, or more generally over {\it any} acceptable family of maximal cubic orders (allowing us to consider families of cubic fields satisfying suitable infinite sets of local conditions). 

We prove Theorems~\ref{cgvsncg} and \ref{diff} and their
corollaries as an~application of a composition law on a prehomogeneous
vector space that was investigated in~\cite{Bhargava1}.  
We lay down the relevant preliminaries about this law of composition,
defined on pairs of ternary quadratic forms, in Section~2; more precisely, we give a correspondence between pairs of ternary quadratic forms and 2-torsion elements in class groups of cubic rings over a principal ideal domain. This allows us to describe a composition law on certain integer-matrix pairs of ternary quadratic forms in
Section~3, leading to groups
that are
closely related to both the class groups and (with some additional
effort) the narrow class groups of orders in cubic fields. In Section~4, we focus on the composition law on \emph{reducible} pairs of ternary quadratic forms,
%---i.e., those that have a common nontrivial rational zero---
which leads to groups that are closely related to the ideal groups of orders in cubic fields. In
Section~5, we describe how one can count the appropriate 
integer points in this
prehomogeneous vector space using the results of~\cite{density}.
By counting these configurations of points, we then complete
the proofs of Theorems~\ref{cgvsncg} and~\ref{diff} in Section~6.  Finally, we 
deduce Corollaries~\ref{odd} and \ref{nocontrast}--\ref{mixed2} in Section~7.

\section{Parametrization of order 2 ideal classes in cubic orders}

The key ingredient in the proof of Theorem~\ref{cgvsncg} is a
parametrization of ideal classes of order~2 in cubic rings by means of
equivalence classes of pairs of integer-matrix ternary quadratic forms, 
which was 
obtained
in~\cite{Bhargava1}.  We begin by describing this orbit space briefly.

If $T$ is a principal ideal domain, then let $V_T=T^2\otimes\Sym^2 T^3$ denote the space of pairs $(A,B)$ of
ternary quadratic forms over $T$.  We write an element
$(A,B)\in V_T$ as a pair of $3\times 3$ symmetric matrices with coefficients in $T$ as
follows:
\begin{equation}\label{expandAB}
(A,B)=\left(
\left[\begin{array}{ccc}{a_{11}}&a_{12}&a_{13}\\a_{12}&a_{22}&{a_{23}}\\a_{13}
&{a_{23}}&{a_{33}}\end{array}\right],
\left[\begin{array}{ccc}{b_{11}}&b_{12}&b_{13}\\b_{12}&b_{22}&{b_{23}}\\b_{13}
&{b_{23}}&{b_{33}}\end{array}\right]
\right).
\end{equation}
In particular, the elements of $V_\Z$ are called pairs $(A,B)$ of {\it integer-matrix} ternary quadratic forms.

The group $\GL_2(T)\times\GL_3(T)$ acts naturally on the space
$V_T$. Namely, an element $g_2\in \GL_2(T)$ acts by changing the
basis of the rank~2 $T$-module of forms spanned by $(A,B)$; more precisely, if
$g_2=\bigl(\begin{smallmatrix}r&s\\t&u\end{smallmatrix}\bigr)$, then we set 
$g_2\cdot(A,B)=(rA-sB,-tA+uB)$.  Similarly, an element $g_3\in
\GL_3(T)$ changes the basis of the rank 3 $T$-module on which
the forms $A$ and $B$ take values; we have $g_3\cdot(A,B)=(g_3Ag_3^t,
g_3Bg_3^t)$. It is easy to see that these actions of $g_2$ and $g_3$ commute, leading to an action of $\GL_2(T)\times\GL_3(T)$. However, this action is not faithful in general.  If we let 
	$$G_T:=\GL_2(T)\times \GL_3(T)/\left\{\left(\left[\begin{smallmatrix}\lambda^{-2} & 0 \\ 0 & \lambda^{-2}\end{smallmatrix}\right],\left[\begin{smallmatrix}\lambda & 0 & 0 \\ 0 & \lambda & 0\\ 0 & 0 & \lambda\end{smallmatrix}\right]\right) : \lambda \in T^\times\right\},$$
then for all principal ideal domains $T$, the group $G_T$ acts faithfully on $V_T$.

The action of $\SL_2(\C)\times\SL_3(\C)$ on $V_\C$ turns out to have a {unique}
polynomial invariant (see, e.g., Sato--Kimura~\cite{SatoKimura}); i.e.,
the ring of all polynomial functions on $V_\C$ that are invariant
under the action of $\SL_2(\C)\times\SL_3(\C)$ is generated by one element.  This generating element can be defined naturally over $\Z$, or indeed, over any principal ideal domain $T$. 

To construct this fundamental
invariant, we notice first that the action of $\SL_3(\C)$ on $V_\C$ has
four independent polynomial invariants, namely the coefficients
$a,b,c,d$ of the binary cubic form
\begin{equation}\label{fdef}
f(x,y) = f_{(A,B)}(x,y) :=\Det(A x - B y),
\end{equation}
where $(A,B)\in V_\C$.
For an element $(A,B)\in V_T$, we call $f(x,y)$ as defined by (\ref{fdef}) the {\it binary cubic form invariant} of $(A,B)$. An element $\gamma \in \GL_2(T)$ then acts on the binary cubic form invariant $f(x,y)$ by $\gamma\cdot f(x,y) := f((x,y)\gamma)$. 

Now it is well-known
that the action of $\SL_2(\C)$ on the space of complex binary cubic forms has a unique polynomial invariant (up to scaling), namely its
discriminant $\Disc(f)$.  
The unique polynomial invariant for the
action of $\SL_2(\C) \times \SL_3(\C)$ on $V_\C$ is then given by $\Disc(A,B):=\Disc(\Det(A x - B y))$, which is an integer polynomial of degree 12 in the entries of $A$ and $B$.

For $(A,B)\in V_T$, we call this fundamental invariant $\Disc(A,B)$ the {\it discriminant} of $(A,B)$. Note that for any 
$(g_2,g_3) \in \GL_2(T) \times \GL_3(T)$ and $(A,B)\in V_T$, we have
$$\Disc((g_2,g_3)\cdot (A,B)) = \det(g_2)^6 \det(g_3)^8 \, \Disc(A,B),$$ 
and thus $\Disc(A,B)$ is also a \emph{relative invariant} for the action of $G_T$ on~$V_T$.  
We say that $(A,B)\in V_T$ is {\it nondegenerate} if it
has nonzero discriminant.

The orbits of $G_T$ on $V_T$ have an important arithmetic
significance, namely they essentially parametrize order 2 ideal
classes in cubic rings over $T$.  Recall that a {\it cubic ring} over $T$ is any ring with unit
that is free of rank 3 as a $T$-module; for example, an order
in a cubic number field is a cubic ring over $\bZ$.  
In 1964, Delone and Faddeev~\cite{DF} showed that orders in cubic
fields are parametrized by irreducible integral binary cubic forms;
this was recently extended to general cubic rings and general binary
cubic forms by Gan, Gross, and Savin~\cite{GGS} (see also the work of Gross and Lucianovic~\cite{GL}). 

We have observed above that a pair $(A,B)\in V_T$ of ternary quadratic forms yields
a binary cubic form invariant $f$ over $T$, and thus we naturally obtain a cubic ring
$R(f)$ over $T$ from an element $(A,B)\in V_T$ via the Gan--Gross--Savin parametrization described above.
We say that a cubic ring over $T$ is {\it nondegenerate} if it has nonzero 
discriminant. 

We now describe a precise correspondence between nondegenerate pairs of ternary
quadratic forms in $V_T$ and ideal classes ``of order 2'' in the corresponding cubic 
rings over a principal ideal domain $T$.  To simplify the statement, we assume that we have a set $S$ of representatives in~$T\setminus\{0\}$ for the left action of $T^\times$ on $T\setminus\{0\}$.  For example, if $T=\Z$, then we will always take $S$ to be the set of positive elements of $\Z$; similarly, if $T=\Z_p$, then $S$ will always  be chosen to consist of the powers of $p$ in~$\Z_p$; and if $T$ is a field, then we will always take $S=\{1\}$.  For an ideal $I$ of a cubic ring $R$ over~$T$, we use $N_S(I)$ to denote the unique generator in~$S$ of the 
usual ideal norm of $I$ viewed as a $T$-ideal.  (When the choice of $S$ is clear, as in the case of $\Z$, $\Z_p$,  or a field, then we usually drop it from the notation.)
Then we have the following generalization of~\cite[Thm.~4]{Bhargava1}:

\begin{theorem}\label{potqfideal} 
Let $T$ and $S$ be as above.  Then
there is a bijection between the set of nondegenerate $G_T$-orbits on
the space $V_T=T^2\otimes\Sym^2 T^3$
and the set of equivalence classes of triples $(R,I,\delta)$, where
$R$ is a nondegenerate cubic ring over~$T$, 
\,$I$ is an ideal of $R$ having rank $3$ as a $T$-module, and $\delta$ is an invertible element of $R\otimes\Frac(T)$ such that
$I^2\subseteq (\delta)$ and $N_S(I)^2\!=\!N(\delta)$.  $($Here two triples
$(R,I,\delta)$ and $(R',I',\delta')$ are equivalent if there exists an 
isomorphism $\phi:R\rightarrow R'$ and an element 
$\kappa\in R'\otimes\Frac(T)$ such that $I'=\kappa \phi(I)$ and $\delta'=\kappa^2\phi(\delta).)$
\end{theorem}

The proof of Theorem~\ref{potqfideal} is identical to that given in \cite{Bhargava1} (with a suitable change to the definition of orientation, described below).  We briefly sketch the bijective map for general $T$. 

Given a triple $(R,I,\delta)$ over $T$, we construct the corresponding pair of ternary quadratic forms over $T$ as follows. First, we may write $R = T + T\omega + T\theta$ where $\langle 1, \omega, \theta \rangle$ is a \emph{normal} basis, i.e., $\omega \theta \in T$. Furthermore, we can write $I = T\alpha_1 + T\alpha_2 + T\alpha_3$, where $\langle \alpha_1 ,\alpha_2,\alpha_3\rangle$ has the \emph{same orientation} as $\langle 1, \omega, \theta\rangle$, i.e., the change-of-basis matrix from $\langle 1, \omega, \theta \rangle$ to $\langle \alpha_1 ,\alpha_2,\alpha_3\rangle$ has determinant lying in our fixed choice $S$ of representatives in $T\setminus\{0\}$ for the left action of~$T^\times$ on~$T\setminus\{0\}$.
Since~$I^2 \subseteq \delta R$, we have for all $i,j \in \{1,2,3\}$ that
	\begin{equation}\label{above}
	\alpha_i\alpha_j = \delta(c_{ij} + b_{ij}\omega + a_{ij}\theta)
	\end{equation}
where $a_{ij}, b_{ij}, c_{ij} \in T$. Then $(A,B)=\left((a_{ij}),(b_{ij})\right)$ yields the desired pair of symmetric $3 \times 3$ matrices over $T$.
In basis-free terms, $(A,B)$ corresponds to the symmetric bilinear map of $T$-modules
	\begin{equation}\label{cf}
	\varphi:I \otimes I \to R/T \qquad \alpha \otimes \beta \mapsto \left(\frac{\alpha \beta}{\delta}\right) \bmod T.
	\end{equation}

On the other hand, if $(A,B) \in V_T$ is a pair of ternary quadratic forms over $T$, 
and $f_{(A,B)}(x,y) = \Det(Ax - By) = ax^3 + bx^2y + cxy^2 + dy^3$ is the associated binary cubic form with coefficients $a,b,c,d \in T$, then $R = R(f_{(A,B)})$ is the cubic ring with basis $\langle 1, \omega, \theta \rangle$ satisfying
	\begin{equation}\label{cubicbasis}
	\begin{array}{rcl}
	\omega \theta &=& -ad \\
	\omega^2 &=& -ac - b\omega + a \theta \\
	\theta^2 &=& -bd - d\omega + c\theta.
	\end{array}
	\end{equation} 
In order to describe $I$ in terms of a basis $\langle \alpha_1, \alpha_2, \alpha_3 \rangle$, it suffices to show that all $c_{ij}$ and $\delta$ in (\ref{above}) are uniquely determined by $(A,B)$. First, 
one shows using Equation (\ref{above}), together with the associative law on triple products $(\delta^{-1} \alpha_i)(\delta^{-1} \alpha_j)(\delta^{-1} \alpha_k)$, that the $c_{ij}$ can be written explicitly as 
	$$c_{ij} = \sum_{{i' < i''}\atop{j' < j''}} 
				\left(\begin{smallmatrix}i & i' & i'' \\1 & 2 & 3\end{smallmatrix}\right) \cdot
				\left(\begin{smallmatrix} j & j' & j'' \\1 & 2 & 3\end{smallmatrix}\right) \cdot	\left|\begin{array}{cc}a_{ij} & a_{ij'} \\a_{i'j} & a_{i'j'}\end{array}\right|
		\cdot 	\left|\begin{array}{cc}b_{ij} & b_{ij''} \\b_{i''j} & b_{i''j''}\end{array}\right|,$$
where $\left(\begin{smallmatrix} \cdot & \cdot & \cdot \\ 1 & 2 & 3 \end{smallmatrix}\right)$ denotes the sign of the permutation of $(123)$ described by the top row. The system of equations given in (\ref{above}) shows that for each $j \in \{1,2,3\}$, we have
	$$\alpha_1 : \alpha_2 : \alpha_3 = c_{1j} + b_{1j}\omega + a_{1j} \theta:c_{2j} + b_{2j}\omega + a_{2j} \theta:c_{3j} + b_{3j}\omega + a_{3j} \theta.$$
The ratios are independent of the choice of $j$, and so this determines $\langle \alpha_1,\alpha_2,\alpha_3\rangle$ uniquely up to a scalar factor in $R \otimes \Frac(T)$. Once we fix a choice of $\langle \alpha_1, \alpha_2, \alpha_3 \rangle$, Equation (\ref{above}) then uniquely determines $\delta\in R \otimes \Frac(T)$. The $R$-ideal structure of the $T$-module $I$ is given by \cite[Eqn.~(11)]{Bhargava1}. This completely and explicitly determines the triple $(R,I,\delta)$ from the pair $(A,B)$ of ternary quadratic forms. 

Finally, we describe the action of $\GL_2(T) \times \GL_3(T)$ on the bases of $R/T$ and $I$ corresponding to the action of the same group on pairs $(A,B)$ of ternary quadratic forms over $T$ described previously. If $g_2 \in \GL_2(T)$ is the matrix $g_2 = \left(\begin{smallmatrix} r & s \\ t & u \end{smallmatrix}\right)$, then it acts on the normal basis of $R/T = \langle \omega, \theta \rangle$  and the $T$-basis $\langle \alpha_1, \alpha_2,\alpha_3\rangle$ of the ideal $I$ by
\begin{eqnarray}\label{g2action}
\langle \omega,\theta \rangle &\mapsto& \det(g_2) \cdot \langle \omega, \theta \rangle \cdot g_2^t= \det(g_2)\cdot\langle r \omega + s \theta, t\omega + u\theta\rangle \\ \nonumber
\langle \alpha_1, \alpha_2,\alpha_3 \rangle &\mapsto& \det(g_2)\cdot\langle \alpha_1, \alpha_2, \alpha_3\rangle. 
\end{eqnarray}
On the other hand, $g_3 \in \GL_3(T)$ acts on the bases $\langle \omega, \theta \rangle$ and $\langle \alpha_1, \alpha_2,\alpha_3\rangle$ by 
\begin{eqnarray}\label{g3action}
\langle \omega,\theta \rangle &\mapsto& \det(g_3)^2\cdot\langle \omega,\theta\rangle \\ \nonumber
\langle \alpha_1, \alpha_2,\alpha_3 \rangle &\mapsto& \langle \alpha_1, \alpha_2, \alpha_3\rangle \cdot g_3^t.
\end{eqnarray}
One now
 easily 
checks that the equivalence defined on triples $(R,I,\delta)$ in the statement of the theorem exactly corresponds to $G_T$-equivalence on pairs $(A,B)\in V_T$ of ternary quadratic forms over~$T$.

\begin{remark} {\em 
Note that $G_\bZ \cong \GL_2(\bZ)\times \SL_3(\bZ)$, and so we recover \cite[Thm.~4]{Bhargava1}. 
}\end{remark}
The following corollary of 
Theorem~\ref{potqfideal} 
describes the 
stabilizer in $G_T$ of an element $(A,B) \in V_T$:

\begin{corollary}\label{stabcor}
The stabilizer in $G_T$ of a
nondegenerate element $(A,B)\in V_T$ is given by the 
semidirect product
\[\Aut(R;I,\delta)\ltimes U_2^+(R_0),\] 
where: $(R,I,\delta)$ is the triple corresponding
to $(A,B)$ as in Theorem~$\ref{potqfideal}$; $\Aut(R;I,\delta)$ is the
group $\{\phi\in\Aut(R) : \exists\kappa\in R\otimes\Frac(T) \mbox{ s.t. }
\phi(I)=\kappa I \mbox{ and } \phi(\delta)=\kappa^2\delta\}$;
$R_0=\End_R(I)$ is the endomorphism ring of $I$; and
$U_2^+(R_0)$ denotes the group of units of $R_0$ having order
dividing~$2$ and norm~$1$.
\end{corollary}

\begin{proof}
First, note that the elements of $\GL_2(T) \times \GL_3(T)$ that preserve the bases of $R/T$ and $I$ under the action described by (\ref{g2action}) and (\ref{g3action}) are exactly the elements of 
	$$C(T):=\left\{\left(\left[\begin{smallmatrix}\lambda^{-2} & 0 \\ 0 & \lambda^{-2}\end{smallmatrix}\right],\left[\begin{smallmatrix}\lambda & 0 & 0 \\ 0 & \lambda & 0\\ 0 & 0 & \lambda\end{smallmatrix}\right]\right) : \lambda \in T^\times\right\}.$$ 
That is, $C(T)$ acts pointwise trivially on $V_T$. 
This implies that if $g \in \GL_2(T) \times \GL_3(T)$, 
then any element of the coset $g\cdot C(T)$ acts on the bases of $R/T$ and $I$ in the same way that $g$ acts. Furthermore, if the actions of $g, g' \in \GL_2(T) \times \GL_3(T)$ are distinct on the bases of $R/T$ and $I$,  then $g' \notin g\cdot C(T)$. 
Thus, it suffices to compute the number of distinct actions on the bases of $R/T$ and $I$ by elements $g\in \GL_2(T) \times \GL_3(T)$ that preserve $A=(a_{ij})$ and $B=(b_{ij})$ in (\ref{above}), in order to compute the number of cosets $g \cdot C(T)$ in the stabilizer in $\GL_2(T) \times \GL_3(T)$ of an element~$(A,B)\in V_T$. 

Now suppose $g = (g_2,g_3) \in \GL_2(T) \times \GL_3(T)$ preserves a pair $(A,B)$.  In terms of the corresponding triple $(R,I,\delta)$, $g$ acts on the basis of $R/T$ as described in        
Equations (\ref{g2action}) and (\ref{g3action}), which induces an automorphism $\phi$ of $R$.
As $g$ must preserve the triple $(R,I,\delta)$ up to equivalence, we see that
$\phi \in \Aut(R;I,\delta)$, i.e., there exists $\kappa\in R\otimes\Frac(T)$ such that 
$\phi(I)=\kappa I$ { and } $\phi(\delta)=\kappa^2\delta$. 
To preserve the structure coefficients $b_{ij}$ and $a_{ij}$ in (\ref{above}), the 
change of basis of $R/T$ corresponding to the element $\phi\in\Aut(R;I,\delta)$ then uniquely determines
the change of basis for $I$ (namely, $\alpha\mapsto\kappa^{-1}\phi(\alpha)$ for $\alpha\in I$, assuming we keep $\delta$ the same in (\ref{above})), up to transformations on the basis of $I$ that act via multiplication by an element of $U_2^+(\End_R(I)$). This is the desired conclusion.
\end{proof} \vspace{-.1in}

\section{Composition of pairs of integer-matrix forms
 and (narrow) 2-class groups}

Let us now restrict ourselves to those cubic rings $R$ over the base ring $\bZ$ that are orders
in $S_3$-cubic fields.  It is known that, when ordered by absolute
discriminant, such $S_3$-cubic orders constitute a proportion of
$100\%$ of all orders in cubic fields (since orders in abelian cubic
fields are negligible~in number, \;\!\!in comparison).  \!Note that the automorphism group $\Aut(R)$ of an $S_3$-cubic order $R$ is~trivial.

Let us say that a triple $(R,I,\delta)$ is {\it projective} if $I$ is projective as an
$R$-module (i.e., if $I$ is invertible as a fractional ideal of $R$); in such a 
case we have $I^2=(\delta)$.  For a fixed $S_3$-cubic order~$R$, 
we therefore obtain a natural law of composition on 
equivalence classes of projective triples $(R,I,\delta)$, defined by
\[
        (R,I,\delta)\circ(R,I',\delta') = (R,II',\delta\delta').
\]
The equivalence classes of projective triples $(R,I,\delta)$ then form a
group under this composition law, which we denote by $H(R)$.  

Let us say that a pair $(A,B)\in V_\Z$ is {\it projective} if the corresponding $(R,I,\delta)$, under the
bijection of Theorem~\ref{potqfideal}, is projective.  We then also obtain a 
composition law on projective equivalence classes of pairs $(A,B)$ of ternary
quadratic forms having binary cubic form invariant equal to a given $f$, where
$R(f)\cong R$ (a higher degree analogue of Gauss composition).  
We denote the corresponding group on such equivalence classes of $(A,B)$ also 
by $H(R)$. 

Fix an $S_3$-cubic order $R$. Let $U$ denote the unit group of $R$, and let
$U^{\mbox{\footnotesize pos norm}}$ denote the subgroup of those units having
positive norm.  Then we have an exact sequence
\begin{equation}\label{hr}
1 \to \frac{U^{\mbox{\footnotesize pos norm}}}{U^{2}} \to H(R) \to \Cl_2(R) \to 1,
\end{equation}
where $U^2$ denotes the subgroup of $U$ consisting of square units. 

If $R$ is an order in a totally real cubic field, then the group
${U^{\mbox{\footnotesize pos norm}}}/{U^2}$ has order $4$. Meanwhile,
if $R$ is an order in a complex cubic field, then the latter group has
order 2.  We thus obtain:

\begin{lemma}\label{hr2}
We have $|H(R)|=4 \cdot|\Cl_2(R)|$ when $R$ is an order in a totally
real $S_3$-cubic field, and 
$|H(R)|=2 \cdot|\Cl_2(R)|$ when $R$ is an order in a complex cubic
field. 
\end{lemma}
Equation (\ref{hr}) and Lemma~\ref{hr2} thus make precise the
relationship between $H(R)$ and $\Cl_2(R)$.

If we further assume that $R$ is an order in a totally real $S_3$-cubic field, 
we can restrict to the group
$H^+(R)\subset H(R)$ consisting of those triples $(R,I,\delta)$ in
which $\delta$ is totally positive\footnote{An element $\delta \in R\otimes \bQ$ is \emph{totally positive} if for every embedding $\sigma : R\otimes \bQ \to \R$, $\sigma(\delta)$ is positive.}.  The group $H^+(R)$ turns out to be
closely related to the group $\Cl_2^+(R)$; in particular, we find:

\begin{lemma}\label{hrp}
  Let $R$ be an order in a totally real $S_3$-cubic field, 
and let $H^+(R)$ be the
  subgroup of $H(R)=\{(R,I,\delta)\}$ where $\delta$ is totally
  positive.  Then $|H^+(R)| = |\Cl_2^+(R)|$.
\end{lemma}

\begin{proof}
  Let $U$ again denote the unit group of $R$, and let $U^{\gg0}\subset
  U$ denote the subgroup of totally positive units. Let
  $\mbox{sgn}:U\to\{\pm1\}^3$ denote the signature homomorphism that
  takes a unit to its sign (one sign for each of the three real
  embeddings $R\rightarrow\R$). Then we have the following commutative
  triangle of exact sequences:
\begin{center}
\includegraphics{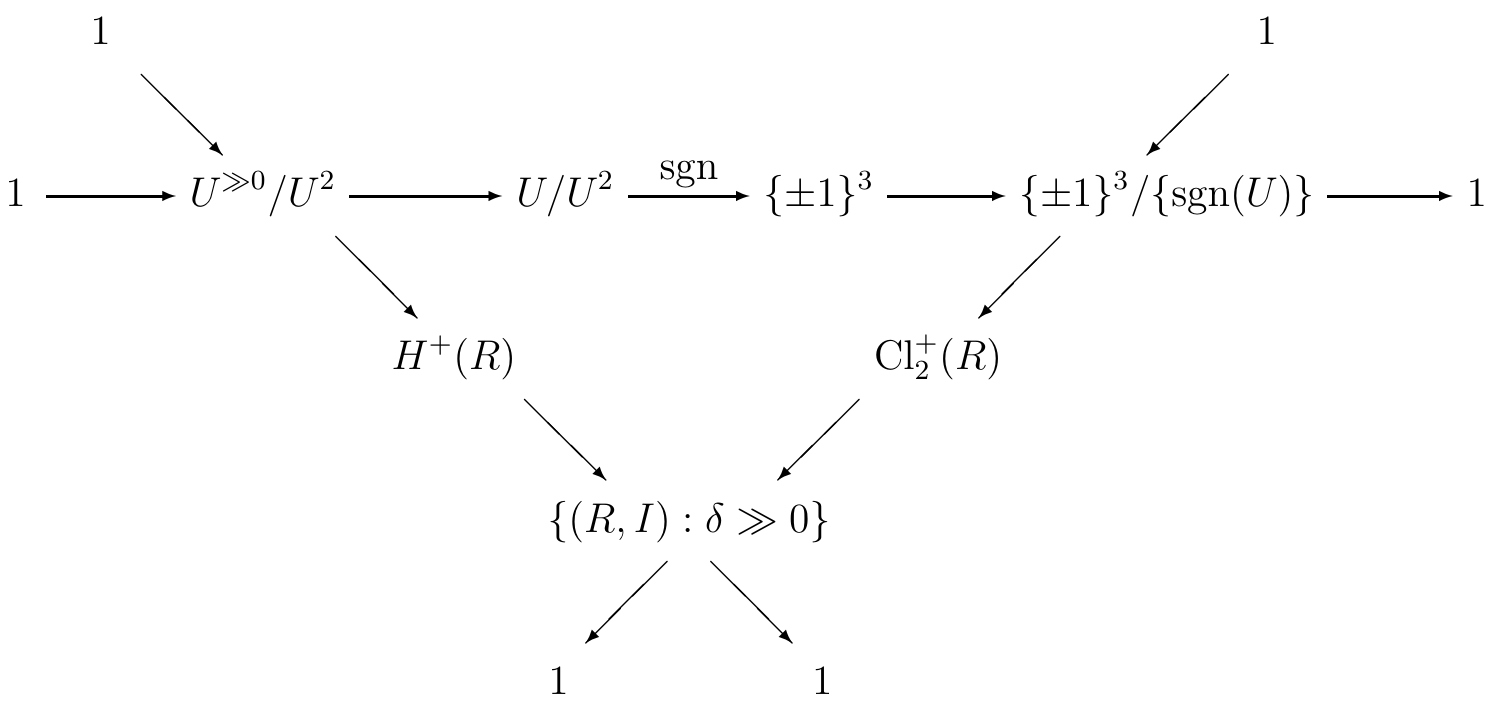} \end{center} 
%\begin{diagram}
% &	&1		&	&				&		&	&			&					&			&		&		&&								& 1		\\
% &	&		&\rdTo&				&		&	&			&					&			&		&		&&\ldTo							& 			\\
% &1	&\rTo	&	& U^{\gg0}/U^2	& \rTo	&	& U/U^2 	& \rTo^{\mbox{sgn}}	&\{\pm1\}^3	&\rTo	&		&\{\pm1\}^3/\{\mbox{sgn}(U)\}	&	&\rTo	& 1 \\
% %&	&		&	&				&		&	&			&					&			&		&		&								&	&		&\\
% &	& 		& 	& 				&\rdTo	&	&			&					&			&		&\ldTo	&								&	&		&\\
% &	&		&	&				&		&H^+(R)& 		&					&			&\Cl_2^+(R)		&		&								&	&		&&\\
% &	&		&	&				&		&	&\rdTo		&	 				&\ldTo		&		&		&								&	&		&\\
% &	&		&	&				&		&	&			&\{(R,I):\delta\gg0\}			&		&		&		&								&	&\\
% &	&		&	&				&		&	&\ldTo		&					& \rdTo		&			&		&		&								&	&		&\\
% &	&		&	&				&		&1	&			& 					&			& 1		&		&								&	&		&	
% \end{diagram}
Here $\{(R,I):\delta\gg0\}$ denotes the group of all equivalence
classes of ideals $I$ of $R$ for which there exists a totally positive
$\delta\in R$ with $I^2=(\delta)$.  Since $|U/U^2|=|\{\pm1\}|^3=8$, we
then also have 
$|U^{\gg0}/U^2|=|\{\pm1\}^3/\{\mbox{sgn}(U)\}|$, and therefore
$|H^+(R)|=|\Cl_2^+(R)|$, as desired.
\end{proof}

\section{Reducible pairs of ternary quadratic forms}

The bijection given in Theorem~\ref{potqfideal} includes all pairs $(A,B) \in V_\Z$ of integer-matrix ternary quadratic forms---even \emph{reducible} {pairs}, i.e., those that share a common rational zero in $\P^2$. The question then arises: which elements in $H(R)$ correspond to reducible pairs $(A,B)$ of ternary quadratic forms? 
\begin{lemma}\label{lemma2}
  Let $(A,B)$ be a {\em projective} element of $V^{(i)}_\Z$ whose binary 
cubic
  form invariant $f_{(A,B)}$ is irreducible over $\Q$, and let $(R,I,\delta)$
  denote the corresponding triple as given by
  Theorem~$\ref{potqfideal}$.  Then $(A,B)$ has a rational zero in
  $\P^2$ if and only if $\delta$ is a square in $(R\otimes\Q)^\times$.
\end{lemma}

\begin{proof}
  Suppose $\delta=r^2$ for some invertible
  $r\in R\otimes\Q$.  Then by replacing $I$ by $r^{-1}I$ and $\delta$
  by $r^{-2}\delta$ if necessary, we may assume that $\delta=1$. Let $\alpha_1$ be the smallest positive element of $I \cap \bZ$, and extend to a basis $\langle \alpha_1, \alpha_2, \alpha_3 \rangle$ of $I$. By (\ref{above}), we have that $b_{11} = 0$ and $a_{11} = 0$, which implies that the associated pair of ternary quadratic forms has a rational zero, namely $(1:0:0)$, in $\bP^2$. 
  
  Conversely, suppose that $A$ and $B$ have a common rational zero in $\bP^2$.  If $(R,I,\delta)$
  is the corresponding triple, then this means that there exists $\alpha \in I$ such that $\phi(\alpha \otimes \alpha)=\alpha^2/\delta \in\Z$, where $\phi$ is defined as in (\ref{cf}).  Setting $\alpha^2 =
  n\delta$ for some $n\in\Z$, and taking norms to $\Z$ on both sides,
  reveals that $N(\alpha)^2=n^3N(\delta)=n^3N(I)^2$.  
Thus $n=m^2$ is a square.
This implies  that $\delta$ must be a square in $(R\otimes\Q)^\times$,
namely, $\delta=(\alpha/m)^2$.
\end{proof}

For a cubic order $R$, let $H_{\red}(R)$ be the subgroup of (equivalence classes of) projective triples $(R, I, \delta) \in H(R)$ where $\delta$ is a square in $(R \otimes \bQ)^\times$. By the previous lemma, $H_{\red}(R)$ is the subgroup of equivalence classes of projective pairs $(A,B)\in V_\Z$ for which $R(f_{(A,B)}) = R$ in their corresponding triple, and such that $(A,B)$ has a rational zero in $\bP^2$. As in the introduction, let $\cI_2(R)$ denote the 2-torsion subgroup of the ideal group of $R$, i.e., the subgroup of invertible fractional ideals $I$ of $R$ such that $I^2 = R$. We can then define a map
	$$\varphi: \cI_2(R) \rightarrow H(R) \qquad I \mapsto (R, I, 1).$$
It is evident that $\im(\cI_2(R)) \subseteq H_\red(R)$. We show that $\varphi$ in fact defines an isomorphism between~$\cI_2(R)$ and $H_\red(R)$:

\begin{theorem}\label{redthm}
The map $\varphi$ yields an isomorphism of $\cI_2(R)$ with $H_\red(R)$.
\end{theorem}
\begin{proof} The preimage of the identity $(R,R,1) \in H(R)$ can only contain 2-torsion ideals of the form $\kappa \cdot R$ where $\kappa \in (R \otimes \bQ)^\times$. To be a 2-torsion ideal, we must have $(\kappa R)^2 = R$; thus $\kappa^2 \in R^\times$ and so $\kappa \in R^\times$. Therefore, the preimage of the identity is simply the ideal $R$ and the map is injective. To show surjectivity, let $(R, I, \delta) \in H_\red(R)$. Since $\delta$ is a square by definition, let $\delta = \xi^2$ and recall that $(R,I,\delta) \sim (R, \xi^{-1}I, 1)$; thus $\xi^{-1}I \in \cI_2(R)$.
\end{proof}

\begin{corollary}\label{maxcase}
Assume that $R$ is maximal. Then $H_\red(R)$ contains only the identity element of $H(R)$, which can be represented by $(R,R,1)$. 
\end{corollary}
\begin{proof} Since maximal orders are Dedekind domains, the only ideal that is 2-torsion in the ideal group of $R$ is $R$. 
\end{proof}

\section{Class numbers of pairs of ternary quadratic forms}

To prove Theorem~\ref{cgvsncg}, we would like to restrict the elements
of $V_\Z$ under consideration to those that are ``irreducible'' in an
appropriate sense.
More precisely, we call a pair $(A,B)$ of
integral ternary quadratic forms in $V_\Z$ {\it absolutely
  irreducible} if \begin{itemize} \ritem{(i)} $A$ and $B$ do not
  possess a common $\Q$-rational zero as conics in $\P^2$; and
  \ritem{(ii)} the binary cubic form $f(x,y)=\Det(Ax-By)$ is
  irreducible over $\Q$.  \end{itemize} 
 
Next, we note that a pair $(A,B)\in V_\R$ of real ternary quadratic
forms gives a pair of conics in $\P^2(\C)$ which, at least in
nondegenerate cases, intersect in four points in $\P^2(\C)$.  We call
these four points the {\it zeroes} of $(A,B)$.  The action of 
$G_\R$ on
$V_\R$ is seen to have three distinct orbits of nondegenerate
elements, namely the orbits $V^{(i)}\subset V_\R$ for $i=0,1,2$, where
$V^{(i)}$ consists of the elements in $V_\R$ having $4-2i$ real zeroes
and $i$ pairs of complex zeroes in $\P^2(\C)$. 

Now let $S\subset V_\Z$ be any $G_\Z$-invariant
subset defined by finitely many congruence conditions. For each prime
$p$, let $S_p$ denote the $p$-adic closure of $S$ in
$V_{\Z_p}=V_\Z\otimes\Z_p$, and let
$M_p(S)$ denote the ``$p$-adic mass'' of $S_p$ in $V_{\Z_p}$, defined
by
\begin{equation}\label{massdef}
 M_p(S) := \sum_{x\in G_{\Z_p}\backslash S_p}
\frac{1}{\Disc_p(x)}\cdot\frac{1}{|\Stab_{G_{\Z_p}}(x)|},
\end{equation}
where $\Disc_p(x)$ denotes the discriminant of $x\in V_{\Z_p}$ as a
power of $p$, and $\Stab_{G_{\Z_p}}(x)$ denotes the stabilizer of $x$ 
in $G_{\Z_p}$.

Setting $n_i=24,4,8$ for $i=0,1,2$ respectively, we may now state the
following result counting the number of absolutely irreducible
elements $(A,B)\in S\cap V^{(i)}_\Z$, up to $G_\Z$-equivalence, having absolute
discriminant at most $X$.

\begin{theorem}\label{cna} For any $G_\Z$-invariant
subset $S\subset V_\Z$ defined by 
finitely many congruence conditions, let 
  $N^{(i)}(S;X)$ denote the number of $G_\Z$-equivalence classes of
  \emph{absolutely irreducible} elements $(A,B)\in S\cap V_\Z^{(i)}$ satisfying
  $|\Disc(A,B)|<X$.  Then 
\begin{equation}\label{ramanujan}
\lim_{X\to\infty} \frac{N^{(i)}(S;X)}X
\,\,=\,\,
\frac{1}{2n_i}\cdot
\prod_p\Bigl(\frac{p-1}{p}\cdot M_p(S)\Bigr).
\end{equation}
\end{theorem}

\begin{proof}
Let $\mu_p(S)$ denote 
the $p$-adic density of the $p$-adic closure $S_p$ of $S$ in $V_{\Z_p}$,
where we normalize the additive measure $\mu$ on $V_{\Z_p}$ so that
$\mu(V_{\Z_p})=1$.  Then \cite[Eqn.~(32)]{density} implies\footnote{Note that in \cite{density}, $V_\bZ$ is defined as the set of all \emph{integer-coefficient} pairs of ternary quadratic forms (see \cite[Eqn.~(3)]{density}). The lattice of all \emph{integer-coefficient} pairs of ternary quadratic forms contains the lattice of all integer-matrix pairs of ternary quadratic forms as an index 64 sublattice. However, the notion of discriminant in \cite{density} also differs from the one used in this paper; namely, an integer-matrix pair $(A,B) \in V_\bZ$ has discriminant equal to $\Disc(\Det(Ax-By))$ in the current paper, but has discriminant equal to $256 \cdot \Disc(\Det(Ax-By))$ according to \cite[\S2]{density}. This implies that the constant in Equation (\ref{oramanujan}) changes from $\displaystyle \frac{\zeta(2)^2\zeta(3)}{2n_i}$ as in \cite{density} to $\displaystyle\frac{2\cdot\zeta(2)^2\cdot\zeta(3)}{n_i}$ here.}
\begin{equation}\label{oramanujan}
\lim_{X\to\infty} \frac{N^{(i)}(S;X)}X
\,\,=\,\,
\displaystyle{\frac{2\cdot\zeta(2)^2\cdot\zeta(3)}{n_i}}
\prod_p
\mu_p(S)
.
\end{equation}
Theorem~\ref{cna} is thus reduced to re-expressing 
the $p$-adic density $\mu_p(S)$ in terms of 
the $p$-adic mass $M_p(S)$. This is accomplished by the following lemma:

\begin{lemma}\label{ramanujan11}
We have
  $$\displaystyle{\mu_p(S) = \left|{4}\right|_p\cdot\frac{\#G_{\F_p}}{p^{12}}\cdot M_p(S)}.$$
\end{lemma}

\begin{proof}
  We normalize the Haar measure 
$dg$ on the $p$-adic group $G_{\Z_p}$ so
  that $\int_{g\in G_{\Z_p}}dg=\#G_{\F_p} \cdot p^{-12}.$  Since
  $|\Disc(x)|_p^{-1}\cdot dx$ is a $G_{\Q_p}$-invariant measure on
  $V_{\Q_p}$,
we must have for any $v_0\in V_{\Z_p}$ of nonzero discriminant
that
\[ \int_{x\in G_{\Z_p}\cdot v_0} dx 
=c\cdot\int_{g\in G_{\Z_p}/\Stab_{G_{\Z_p}}\!(v_0)} |\Disc(gv_0)|_p\cdot dg 
=c\cdot\frac{\#G_{\F_p}\cdot p^{-12}}{\Disc_p(v_0)\cdot\#\Stab_{G_{\Z_p}}(v_0)},
\]
for some constant $c$.
The constant $c=c(v_0)$ can be determined by calculating the Jacobian of
the change of variable $g\mapsto g\cdot v_0$ from $G_{\Z_p}$ to
$V_{\Z_p}$.  Now this is an algebraic calculation, involving polynomials
in the coordinates on $G_{\Z_p}$ and $V_{\Z_p}$.  Since any two
choices of $v_0$ are $G_K$-equivalent for some finite extension $K$
of $\Q_p$ (since $V_{\Z_p}$ is a prehomogeneous vector
space~\cite{SatoKimura}, 
there is one open orbit over the algebraic closure $\bar\Q_p$ of
$\Q_p$), and since the Haar measure $dg$ naturally extends to a 
Haar measure on $G_K$, we conclude that our constant $c$ cannot depend
on $v_0$.  

It thus suffices to compute the constant $c$ for 
$v_0=\left(
\left[\begin{smallmatrix} 0& 0& 1 \\ 0& 0& 0\\1
&0 & 1\end{smallmatrix}\right],
\left[\begin{smallmatrix}0 & 1 &0 \\ 1& 1 & 0\\ 0
&0 &0 \end{smallmatrix}\right]
\right),$ 
which corresponds to the 
triple $(\Z_p^3,\Z_p^3,1)$. When $p \neq 2$, it can be shown that $G_{\Z_p}\cdot v_0\subset
V_{\Z_p}$ is defined simply by conditions modulo~$p$. Indeed, the only cubic ring over $\Z_p$ whose 
reduction modulo~$p$ is isomorphic to $\F_p^3$ is~$\Z_p^3$, and the only 
ideal class of $\Z_p^3$ having rank 3 over $\Z_p$ is~$\Z_p^3$.  
If $\delta\in\Z_p^3$ reduces to $1=(1,1,1)$ modulo~$p$, then
$\delta$ is in the same squareclass as 
$1\in(\Z_p^3)^\times$, and thus the triple 
$(\Z_p^3,\Z_p^3,\delta)$ is 
equivalent to $(\Z_p^3,\Z_p^3,1)$ for odd primes $p.$ When $p = 2$, if $\delta \in \bZ_2^3$ reduces to $1 = (1,1,1)$ modulo $p$, then there are 16 choices for the squareclass of $\delta \in \bZ_2^3$. The group $\Aut(\Z_2^3)\cong S_3$ acts on these 16 squareclasses, and the squareclass of $\delta=(1,1,1)$ lies in a single orbit. Furthermore, $\Aut(\Z_2^3;\Z_2^3,\delta)$, for $\delta$ in any one of these 16 squareclasses, is simply the stabilizer of that squareclass under this action of $S_3$.  It follows by Corollary~\ref{stabcor} and the orbit-stabilizer theorem that the $2$-adic mass of $G_{\Z_2}\cdot v_0$---i.e., the $2$-adic mass of the $G_{\Z_2}$-orbit of all $(A,B)\in V_{\Z_2}$ corresponding to the triple $(\Z_2^3,\Z_2^3,1)$---is exactly $1/16$ of the $2$-adic mass of the set of all $(A,B)\in V_{\Z_2}$ whose reduction modulo~2 corresponds to the triple $(\F_2^3,\F_2^3,1)$. 

Now the cardinality of $V_{\F_p}$ is $p^{12}$.  Therefore, if $p\neq 2$, then we see that the measure of $G_{\Z_p}\cdot v_0$
in $V_{\Z_p}$ is equal to $p^{-12}$ times the cardinality
of $G_{\F_p}\cdot \bar{v}_0$ in $V_{\F_p}$,
where $\bar{v}_0$ denotes the reduction of $v_0$ modulo $p$; if $p = 2$, then the measure of $G_{\Z_p}\cdot v_0$ is equal to $p^{-12}$ times the cardinality of $G_{\bF_2}\cdot \bar{v}_0$ in $V_{\F_2}$ times $1/16$. 
Since for our choice of $v_0$, we have $U_2^+(\bZ_p^3) = 4 = U_2^+(\bF_p^3)$ for odd primes $p$, and $U_2^+(\bZ_2^3) = 4 = 4\cdot U_2^+(\bF_2^3)$, we conclude using Corollary \ref{stabcor} that $\#\Stab_{G_{\Z_p}}(v_0)=\#\Stab_{G_{\F_p}}(\bar{v}_0)=24$ when $p \neq 2$, and $\#\Stab_{G_{\Z_2}}(v_0)= 4\cdot\#\Stab_{G_{\F_2}}(\bar{v}_0) = 24$. Thus, by the orbit-stabilizer theorem, we obtain
\[ \int_{x\in G_{\Z_p}\cdot v_0} dx = \left|{16}\right|_p\cdot p^{-12} \cdot
\#\bigl(G_{\F_p}\cdot\bar{v}_0\bigr) 
= \left|{16}\right|_p \cdot \frac{\#G_{\F_p}\cdot p^{-12}}{\#\Stab_{G_{\F_p}}(v_0)}
= \left|{4}\right|_p \cdot \frac{\#G_{\F_p}\cdot p^{-12}}{\Disc_p(v_0)\cdot\#\Stab_{G_{\Z_p}}(v_0)},
\]
since $\Disc_p(v_0)=1$ for our choice of $v_0$.  Thus we must have
$c=\left|{4}\right|_p$ for any choice of $v_0$ satisfying
$\Disc(v_0)\neq 0$, and
the lemma follows.
\end{proof}

\noindent 

Returning to the proof of Theorem~\ref{cna}, we observe that $\#G_{\F_p}=(p^2-1)(p^2-p)\cdot \linebreak(p^3-1)(p^3-p)(p^3-p^2)/(p-1)$, and so by Lemma~\ref{ramanujan11}, we have
\[\frac{2\cdot\zeta(2)^2\cdot\zeta(3)}{n_i}\cdot\prod_p\mu_p(S) = 
\frac{\zeta(2)^2\zeta(3)}{2n_i}\cdot\prod_p 
\Bigl(1-\frac1{p^2}\Bigr)\cdot
\Bigl(1-\frac1{p^3}\Bigr)\cdot
\Bigl(1-\frac1{p^2}\Bigr)\cdot
\Bigl(\frac{p-1}{p}\cdot M_p(S)\Bigr),\]
yielding Theorem~\ref{cna}.
\end{proof}

To prove Theorems~\ref{cgvsncg} and \ref{diff}, we choose
appropriate sets $S$ on which to apply
Theorem~\ref{cna}.
This is carried out in the next section.

\section{Proof of Theorems~\ref{cgvsncg} and ~\ref{diff}}

Our goal is to count the $G_\bZ$-orbits of pairs of forms in $V_\bZ^{(i)}$ of bounded absolute discriminant that correspond, under the bijection described in Theorem~\ref{potqfideal}, to equivalence classes of triples $(R,I,\delta)$ where $R$ lies in some {acceptable} family $\Sigma$ of cubic orders and $I$ is projective. More precisely, for each prime $p$, let $\Sigma_p$ be a set of isomorphism classes of
nondegenerate cubic rings over $\Z_p$.  
%(By nondegenerate over $\Z_p$, we 
%mean having nonzero discriminant over $\Z_p$, so that it can arise as
%$R\otimes\Z_p$ for some cubic order $R$ over $\Z$.)  
We denote the
collection $(\Sigma_p)$ of these local specifications over all $p$
by $\Sigma$.  As in the introduction, we say that the collection $\Sigma=(\Sigma_p)$ is {\it
  acceptable} if, for all sufficiently large $p$, the set~$\Sigma_p$
contains all {maximal} cubic rings over $\Z_p$ that are not totally ramified.
Finally, for a cubic order~$R$ over $\Z$, we write ``$R\in\Sigma$''
(or say that ``$R$ is a $\Sigma$-order'') if $R\otimes\Z_p\in\Sigma_p$
for all $p$.  

Now fix any acceptable collection $\Sigma=(\Sigma_p)$ of local
specifications.  Let $S=S({\Sigma})$
denote the set of all absolutely irreducible elements $(A,B)\in
V_\Z$ such that, in the corresponding triple $(R,I,\delta)$, we
have that $R\in\Sigma$ and $I$ is invertible as an ideal class of $R$
(implying that $I\otimes\Z_p$ is the trivial ideal class of
$R\otimes\Z_p$).  Then we wish to count $N^{(i)}(S(\Sigma); X)$, the number of $G_\bZ$-equivalence classes of absolutely irreducible elements $(A,B) \in S(\Sigma) \cap V_{\bZ}^{(i)}$ satisfying $|\Disc(A,B)|<X$.

Although $S$ might be defined by infinitely
many congruence conditions, the estimate provided
in~\cite[Prop.~23]{density} and \cite[Lem.~3.7]{geosieve} (and the fact that $\Sigma$ is
acceptable) implies that Equation~(\ref{ramanujan}) continues to hold
for the set $S$ (this follows from an identical argument as in \cite[\S3.3]{density} or \cite[\S3.4]{geosieve}). More precisely,
\begin{equation}\label{ramanujangen}
\lim_{X\rightarrow \infty} \frac{N^{(i)}(S(\Sigma);X)}{X} = \frac{1}{2n_i} \cdot \prod_p \left(\frac{p-1}{p}\cdot M_p(S(\Sigma))\right).
\end{equation}

We also have the following algebraic lemma which describes when a pair $(A,B)$ lies in $V_\bZ^{(i)}$  ($i\in\{0,1,2\}$) in terms of the associated triple $(R,I,\delta)$:

\begin{lemma}\label{lemma1}
Let $(A,B)$ be an element of 
  $V^{(i)}_\Z$ 
whose binary cubic form invariant $f$ is irreducible  over $\Q$, 
and let $(R,I,\delta)$ be the corresponding triple as
  given by Theorem~$\ref{potqfideal}$. 
\begin{itemize}  
\item[\rm{(a)}]
If $i=0$, then $R$ is an order in a totally real cubic field, and
$\delta$ is a totally positive element of $R\otimes\Q$.
\item[\rm{(b)}]
If $i=1$, then $R$ is an order in a complex cubic field.
\item[\rm{(c)}]
If $i=2$, then $R$ is an order in a totally real cubic field, and
$\delta\in R\otimes\Q$ is not a totally positive element. 
\end{itemize}
\end{lemma}

\begin{proof}
If the binary cubic form invariant $f$ of $(A,B)$ is irreducible over $\Q$,
then the cubic ring $R=R(f)$ is a domain~\cite{DF} (see
also~\cite[Prop.~11]{simple}), 
and thus an order
in a cubic field.  To check the assertions about $\delta$, it suffices
to base change to the real numbers $\R$, where one can simply check
the assertion at one point of each of the three orbits $V^{(i)}$ for
$i=0$, 1, and 2.
\end{proof}	\\
Combining Theorem~\ref{redthm} with Lemmas~\ref{hr2},~\ref{hrp}, and~\ref{lemma1},
we obtain 
\begin{equation}\label{irredcountgen}
\begin{array}{rcl}
N^{(0)}(S(\Sigma), X) + N^{(2)}(S(\Sigma), X) &=& 
	 \displaystyle\sum_{\scriptstyle{R \in \Sigma,}\atop{\scriptstyle 0<\Disc(R)<X}} 4\cdot|\Cl_2(R)| - |\cI_2(R)|,  \\[.425in]
N^{(1)}(S(\Sigma), X) &=&	\displaystyle\sum_{\scriptstyle{R \in \Sigma,}\atop{\scriptstyle -X < \Disc(R) < 0}} \!\!\!2 \cdot |\Cl_2(R)| - |\cI_2(R)|, \\[.425in]
N^{(0)}(S(\Sigma),X) &=&  \displaystyle\sum_{\scriptstyle{R \in \Sigma,}\atop{\scriptstyle 0<\Disc(R)<X}} \;\;\;\;|\Cl^+_2(R)| - |\cI_2(R)| .
\end{array}
\end{equation}

By Theorem~\ref{potqfideal} and Corollary~\ref{stabcor}, 
the $p$-adic masses of $S=S(\Sigma)$ defined in (\ref{massdef}) can 
be expressed as
\begin{equation}\label{massdef2}
 M_p(S(\Sigma)) = \sum \frac{1}{\Disc_p(R)\cdot|\Aut(R;I,\delta)|\cdot|U_2^+(R_0)|},
\end{equation}
where the sum is over all equivalence classes of triples
$(R,I,\delta)$ over $\Z_p$ represented in $S_p$.
If $R\in\Sigma_p$ is a nondegenerate
cubic ring over $\Z_p$, then in a corresponding triple
$(R,I,\delta)$ we can always choose~$I=R$, since $I$ is a principal
ideal (recall that invertible means locally principal).  
The number of elements $\delta$ yielding distinct valid triples $(R,R,\delta)$ over $\bZ_p$ (in the sense of Theorem~\ref{potqfideal}), up to equivalence, is equal to the number of $\Aut(R)$-orbits on the set $U^+(R)/{U^+(R)}^{\times 2}$, where $U^+(R)$ denotes the group of units of $R$ having norm 1. If $\delta \in U^+(R)$ has image $\overline{\delta} \in U^+(R)/{U^+(R)}^{\times 2}$, then the stabilizer of $\overline{\delta}$ under this action of $\Aut(R)$ is given by $\Aut(R;R,\delta)$ for any lift $\delta$ of $\overline{\delta}$. By the orbit-stabilizer theorem, we then obtain that

\begin{equation}\label{massdef3}
 M_p(S(\Sigma)) = \sum \frac{|U^+(R)/U^+(R)^{\times2}|}
{\Disc_p(R)\cdot|\Aut(R)|
\cdot|U_2^+(R)|}
\end{equation}
where the sum is over all isomorphism classes of cubic rings $R$ over
$\Z_p$ lying in $\Sigma_p$. 

\begin{lemma}\label{weird}
Let $R$ be a nondegenerate cubic ring over $\Z_p$.  Then
\[\frac{|U^+(R)/U^+(R)^{\times 2}|}{|U_2^+(R)|}
\]
is $1$ if $p\neq 2$, and $4$ if $p=2$.
\end{lemma}

\begin{proof}
  The unit group of $R$, as a multiplicative group, is isomorphic to
  $F'\times G'$, where $F'$ is a finite abelian group and $G'$ is 
torsion-free and may naturally be viewed as a free
rank 3 $\Z_p$-module.  Hence the submodule $U^+(R)$,
  consisting of those units having norm 1, is isomorphic to $F\times
  G$, where $F$ is a finite abelian group and $G$ is free of rank 2 as a
  $\Z_p$-module.  

  Let $F_2$ denote the 2-torsion subgroup of $F$.  Since $F_2$ is the
  kernel of the multiplication-by-2 map on $F$, it is clear that
  $|F/(2\cdot F)|/|F_2|=1$.  Therefore, it suffices to check the lemma on the
  ``free'' part $G$ of $U^+(R)$, namely, on the $\Z_p$-module $\Z_p^2$,
  where the result is clear.  (The case $p=2$ differs because, while
  $2\cdot\Z_p^2 = \Z_p^2$ for $p\neq 2$, the $\Z_2$-module 
$2\cdot\Z_2^2$ has index 4 in $\Z_2^2$.)
\end{proof}

Combining (\ref{ramanujangen}), (\ref{massdef3}), and
Lemma~\ref{weird}, we obtain 
\begin{equation}\label{ramanujan2}
\lim_{X\to\infty} \frac{N^{(i)}(S(\Sigma);X)}X
=\frac{2}{n_i}\cdot
\prod_p\Bigl(\frac{p-1}{p}\cdot \sum_{R\in\Sigma_p}
\frac{1}{\Disc_p(R)}\cdot\frac1{|\Aut(R)|}\Bigr).
\end{equation}
Let $c_\Sigma$ denote the Euler product occurring 
in (\ref{ramanujan2}), i.e., 
\begin{equation*}\label{csigma}
c_\Sigma := \prod_p\Bigl(\frac{p-1}{p}\cdot \sum_{R\in\Sigma_p}
\frac{1}{\Disc_p(R)}\cdot\frac1{|\Aut(R)|}\Bigr).
\end{equation*}
We now recall the following result from \cite[Thm.~8]{simple}:

\begin{lemma}\label{dht}
\hfill
\begin{itemize}
\item[$($a$)$] The number of totally real $\Sigma$-orders $\O$ with
  $|\Disc(\O)|<X$ is $\displaystyle\frac{1}{12}c_\Sigma\cdot X+o(X)$. 

\item[$($b$)$] The number of complex $\Sigma$-orders $\O$ with $|\Disc(\O)|<X$ is
  $\displaystyle\frac{1}{4} c_\Sigma\cdot X+o(X)$.
\end{itemize}
\end{lemma}   
Thus, we may conclude from (\ref{irredcountgen}),  (\ref{ramanujan2}), and Lemma~\ref{dht} that
\begin{equation}
\begin{array}{rcccl}
\frac{\displaystyle\sum_{{\mbox{\scriptsize $R \in \Sigma$,}}\atop{\mbox{\scriptsize $0<\Disc(R)<X$}}} |\Cl_2(R)| - \displaystyle\frac{1}{4}\cdot |\cI_2(R)|}{\displaystyle\sum_{{\mbox{\scriptsize $R \in \Sigma$,}}\atop{\mbox{\scriptsize $0<\Disc(R)<X$}}} 1} &=& \displaystyle\frac{1}{4}\cdot\frac{\displaystyle\frac{2}{n_{0}}\cdot c_\Sigma + \frac{2}{n_2} \cdot c_\Sigma}{\displaystyle\frac{1}{12}\cdot c_\Sigma} &=& 1, \vspace{.15in}\\	\frac{\displaystyle\sum_{{\mbox{\scriptsize $R \in \Sigma$,}}\atop{\mbox{\scriptsize $-X<\Disc(R)<0$}}} |\Cl_2(R)| - \displaystyle\frac{1}{2}|\cI_2(R)|}{\displaystyle\sum_{{\mbox{\scriptsize $R \in \Sigma$,}}\atop{\mbox{\scriptsize $-X<\Disc(R)<0$}}} 1} &=& \displaystyle\frac{1}{2}\cdot \frac{\displaystyle\frac{2}{n_1} \cdot c_\Sigma}{\displaystyle{\frac{1}{4}}c_\Sigma} &=& 1, \quad \mbox{and} \vspace{.15in}\\
\frac{\displaystyle\sum_{{\mbox{\scriptsize $R \in \Sigma$,}}\atop{\mbox{\scriptsize $0<\Disc(R)<X$}}} |\Cl^+_2(R)| - |\cI_2(R)|}{\displaystyle\sum_{{\mbox{\scriptsize $R \in \Sigma$,}}\atop{\mbox{\scriptsize $0<\Disc(R)<X$}}} 1} &=&  \displaystyle\frac{\displaystyle\frac{2}{n_0} \cdot c_\Sigma}{\displaystyle{\frac{1}{12}}c_\Sigma} &=& 1.
\end{array}
\end{equation}
This proves Theorem~\ref{diff}. Furthermore, when $\Sigma$ is an acceptable collection of maximal cubic orders, then Theorem~\ref{diff} and Corollary~\ref{maxcase} imply Theorem~\ref{cgvsncg}, since $|\cI_2(R)| =1$ when $R$ is maximal.

\section{Proofs of Corollaries}

We conclude by proving Corollaries~\ref{odd}, \ref{nocontrast},
\ref{contrast}, \ref{mixed1}, and \ref{mixed2}.  As remarked in the introduction, we in fact prove a generalization of each of these corollaries.  Namely, let $\Sigma$ be any
acceptable collection of local specifications of maximal cubic orders (thus possibly defined by infinitely many local conditions).  Then we prove that
Corollaries~\ref{odd} and  \ref{nocontrast}--\ref{mixed2} continue to hold, with the same percentages, for the family of maximal orders defined by any such $\Sigma$. 

We begin by demonstrating that in the family of all real (resp.\ complex) cubic fields defined by $\Sigma$, a proportion of at least 75\% (resp.~50\%) have odd class number. 

\vspace{.2in}
\noindent {\bf Proof of Corollary~\ref{odd}:} 
Indeed, suppose a lower density of less than 75\% (resp.\ 50\%) of totally real (resp.\
complex) $\Sigma$-orders have odd class number.  Then
an upper density of more than 1/4 (resp.\ 1/2) of these
$\Sigma$-orders $R$ would satisfy $|\Cl_2(R)|\geq 2$.  Thus the limsup of the average
size of $|\Cl_2(R)|$ would be strictly larger than
$ 1 + (1/4) = 5/4$ \,\,(resp.\ $1+ (1/2) = 3/2)$, 
contradicting Theorem~1(a) (resp.\ Theorem~1(b)). {$\Box$ \vspace{2 ex}}

Next, we show that at least 50\% of real cubic fields $K$ in the family of cubic fields defined by $\Sigma$ satisfy $\Cl_2(K)\neq \Cl_2^+(K)$, and at least 25\% of such $K$ satisfy $\Cl_2(K)= \Cl_2^+(K)$.

\vspace{.2in}
\noindent
{\bf Proof of Corollary~\ref{nocontrast}:}
By the theorem of Armitage and Frohlich, for a cubic field~$K$
we have either $|\Cl_2^+(K)|=|\Cl_2(K)|$ or $|\Cl_2^+(K)|=
2\cdot|\Cl_2(K)|$.  Suppose that a lower density of strictly less than~50\% of the totally real
cubic fields $K$ defined by $\Sigma$ satisfy $|\Cl_2^+(K)|= 2\cdot |\Cl_2(K)|$.  
Then, since $|\Cl_2(K)|\geq 1$ for all~$K$,  the liminf of 
the average size of
$|\Cl_2^+(K)|$ would be strictly less than the average size of
$2\cdot|\Cl_2(K)|-(1/2)$.
However, the average size of
$2\cdot|\Cl_2(K)|-(1/2)$ is 2 by Theorem~1(a), contradicting Theorem
1(c).  {$\Box$ \vspace{2 ex}}

\vspace{.05in}

\noindent
{\bf Proof of Corollary~\ref{contrast}:}
Suppose that an upper density of strictly more than 75\% of
totally real cubic fields $K$ defined by $\Sigma$ satisfy $|\Cl_2^+(K)|\geq
2\cdot|\Cl_2(K)|\geq |\Cl_2(K)|+1$.
Then the limsup of the average size of
$|\Cl_2^+(K)|$ would be strictly larger than the average size of
$|\Cl_2(K)|+ (3/4)$, which is~2.  This again contradicts Theorem~1(c). 
{$\Box$ \vspace{2 ex}}

Finally, we prove that at least 50\% of real cubic fields $K$ in the family of cubic fields defined by $\Sigma$ do not possess units of every possible signature; in addition, at least 75\% of such cubic fields~$K$ possess units of mixed signature.

\vspace{.2in}
\noindent
{\bf Proof of Corollary~\ref{mixed1}:}
We note that if the class group and narrow class group of a number
field~$K$ are not isomorphic, then $K$ cannot possess a unit of every
possible real signature.  Corollary~\ref{nocontrast} now implies the
result. 
{$\Box$ \vspace{2 ex}}

\vspace{.05in}
\noindent
{\bf Proof of Corollary~\ref{mixed2}:}
By Corollary~\ref{odd}, a lower density of at least $75\%$ of cubic fields $K$ have odd
class number.  For any such cubic field~$K$, the 2-Sylow subgroup of
$\Cl^+(K)/\Cl(K)$ will be either trivial or a cyclic group of order 2,
by Armitage and Frohlich's theorem.  Thus any such cubic field~$K$ will
contain a unit of mixed signature.  {$\Box$ \vspace{2 ex}}

\subsection*{Acknowledgments}

We are very grateful to Benedict Gross,
Kiran Kedlaya, J\"urgen Kl\"uners, Hendrik
Lenstra, Peter Sarnak, Arul Shankar, Christopher Skinner, 
and Melanie Wood for helpful discussions.  The first author was supported by the Packard and Simons
Foundations and NSF Grant DMS-1001828. The second author was supported
by a National Defense Science \& Engineering Fellowship and an NSF
Graduate Research Fellowship.

\end{document}